\documentclass[11pt]{article}
\usepackage{amsmath,amsthm,amssymb}
\usepackage{bbm}

\def\ind{\mathbbm{1}}
\def\Ex{\mathbb E}
\def\Pr{\mathbb P}
\def\er{\mathbb R}
\def\Med{\mathrm{Med}}
\def\Var{\mathrm{Var}}
\def\Cov{\mathrm{Cov}}

\newtheorem{thm}{Theorem}
\newtheorem{lem}[thm]{Lemma}
\newtheorem{prop}[thm]{Proposition}

\title{Order statistics and concentration of $l_r$ norms for log-concave vectors
\thanks{AMS 2010 Mathematics Subject Classification: 60E15 (52A38, 60B11).}
\thanks{
Key words and phrases: log-concave measures, order statistics, concentration of
volume.}
\thanks{
Research partially supported by MNiSW Grant no. N N201 397437 and the Foundation for Polish 
Science.}}
\author{Rafa{\l} Lata{\l}a \thanks{Institute of Mathematics, University of Warsaw, 
Banacha 2, 02-097 Warszawa, Poland 
and Institute of Mathematics, Polish Academy of Sciences, 
ul. \'{S}niadeckich 8, 00-956 Warszawa, Poland, e-mail: rlatala@mimuw.edu.pl.}}

\date{}

\begin{document}

\maketitle

\begin{abstract}
We establish upper bounds for tails of order statistics of isotropic log-concave
vectors and apply them to derive a concentration of $l_r$ norms of such vectors.
\end{abstract}

\section{Introduction and notation}

An $n$ dimensional random vector is called log-concave if it has a log-concave distribution, i.e.\
for any compact nonempty sets $A,B\subset \er^n$ and $\lambda\in (0,1)$,
\[
\Pr(X\in \lambda A+(1-\lambda)B)\geq \Pr(X\in A)^{\lambda}\Pr(X\in B)^{1-\lambda},
\]
where $\lambda A+(1-\lambda)B=\{\lambda x+(1-\lambda)y\colon x\in A,y\in B\}$. By the result of 
Borell \cite{Bo} a vector $X$ with full dimensional support is log-concave  if
and only if it has a density of the form $e^{-f}$, where 
$f\colon \er^n\rightarrow (-\infty,\infty]$
is a convex function. Log-concave vectors are frequently
studied in convex geometry, since by the Brunn-Minkowski inequality  uniform distributions
on convex sets as well as their lower dimensional marginals are log-concave. 

A random vector $X=(X_1,\ldots,X_n)$  is isotropic if $\Ex X_i=0$ and
$\Cov(X_i,X_j)=\delta_{i,j}$ for all $i,j\leq n$. Equivalently, an $n$-dimensional random
vector with mean zero is isotropic if $\Ex\langle t,X\rangle^2 =|t|^2$ for any $t\in \er^n$.
For any nondegenerate log-concave vector $X$ there exists an affine transformation $T$ such
that $TX$ is isotropic.

In recent years there were derived numerous important properties of log-concave vectors.
One of such results is the Paouris concentration of mass \cite{Pa} that states that for any
isotropic log-concave vector $X$ in $\er^n$,
\begin{equation}
\label{conc_mass}
\Pr(|X|\geq Ct\sqrt{n})\leq \exp(-t\sqrt{n})\quad \mbox{ for } t\geq 1.
\end{equation}

One of purposes of this paper is the extension of the Paouris result to $l_r$ norms, that
is deriving upper bounds for $\Pr(\|X\|_r\geq t)$, where $\|x\|_r=(\sum_{i=1}^n|x_i|^r)^{1/r}$.
For $r\in [1,2)$ this is an easy consequence of \eqref{conc_mass} and H\"older's inequality,
however the case $r>2$ requires in our opinion new ideas. We show that
\[
\Pr\big(\|X\|_r\geq C(r)tn^{1/r}\big)\leq \exp\big(-tn^{1/r}\big)\quad \mbox{ for } t\geq 1,\ r>2,
\]
where $C(r)$ is a constant depending only on $r$ -- see Theorem \ref{thm_larger}.
Our method is based on suitable tail estimates for order statistics of $X$.
  
For an $n$--dimensional random vector $X$ by $X_1^*\geq X_2^*\geq \ldots\geq X_n^*$ we denote
the nonincreasing rearrangement of $|X_1|,\ldots,|X_n|$ 
(in particular $X_1^*=\max\{|X_1|,\ldots,|X_n|\}$ and
$X_n^*=\min\{|X_1|,\ldots,|X_n|\}$). Random variables $X_k^*$, $1\leq k\leq n$, are called 
order statistics of $X$. 

By \eqref{conc_mass} we immediately get for isotropic, log-concave vectors $X$
\[
\Pr(X_k^*\geq t)\leq \exp\Big(-\frac{1}{C}\sqrt{k}t\Big)
\]
for $t\geq C\sqrt{n/k}$. The main result of the paper is Theorem \ref{estorder} which
states that the above inequality holds for $t\geq C\log(en/k)$ -- as shows the example
of exponential distribution this range of $t$ is for $k\leq n/2$ optimal up to a universal constant.

Tail estimates for order statistics can be also applied to provide optimal
estimates for $\sup_{\#I=m}|P_I X|$, where supremum is taken over all subsets of $\{1,\ldots,n\}$
of cardinality $m\in[1,n]$ and $P_I$ denotes the coordinatewise projection.  
The details will be presented in the forthcoming 
paper \cite{ALLPT}.

The organization of the article is as follows.
In Section 2 we discuss upper bounds for tails of order statistics and their connections
with exponential concentration and Paouris' result. Section 3 is devoted to
the derivation of tail estimates of $l_r$ norms for log-concave vectors.
Finally Section 4 contains a proof of Theorem \ref{estN}, which is a crucial tool used
to derive our main result.

Throughout the article by $C,C_1,\ldots$ we denote 
universal constants. Values of a constant $C$ may differ
at each occurence. For $x\in \er^n$ we put $|x|=\|x\|_2=(\sum_{i=1}^n x_i^2)^{1/2}$.

\section{Tail estimates for order statistics}

If coordinates of $X$ are independent symmetric exponential random variables with variance one
then it is not hard to see that
$\Med(X_k^*)\geq \frac{1}{C}\log(en/k)$ for any $1\leq k\leq n/2$. So we may obtain a reasonable
bound for $\Pr(X_k^*\geq t)$, $k\leq n/2$ in the case of log-concave vectors only for 
$t\geq \frac{1}{C}\log(en/k)$.  Using the idea that exponential random vectors 
are extremal in the class of unconditional log-concave vectors 
(i.e. such vectors that  $(\eta_1X_1,\ldots,\eta_nX_n)$ has the same distribution as $X$ 
for any choice of signs $\eta_i\in \{-1,1\}$) one may easily 
derive the following fact.

\begin{prop}
\label{uncond}
If $X$ is log-concave and unconditional $n$-dimensional isotropic random vector then
\[
\Pr(X_k^*\geq t)\leq \exp\Big(-\frac{1}{C}kt\Big)\quad \mbox{ for }t\geq C\log\Big(\frac{en}{k}\Big).
\]
\end{prop} 

\begin{proof}
The result of Bobkov and Nazarov \cite{BN} implies that for any $i_1<i_2<\ldots<i_k$ and $t>0$
\[
\Pr(|X_{i_1}|\geq t,\ldots,|X_{i_k}|\geq t)=2^k \Pr(X_{i_1}\geq t,\ldots,X_{i_k}\geq t)\leq
2^k\exp\Big(-\frac{1}{C}kt\Big)
\]
Hence 
\begin{align*}
\Pr(X_k^*\geq t)&\leq \sum_{1\leq i_1<\ldots<i_k\leq n}\Pr(|X_{i_1}|\geq t,\ldots,|X_{i_k}|\geq t)\leq
\binom{n}{k}2^k\exp\Big(-\frac{1}{C}kt\Big)
\\
&\leq \Big(\frac{2en}{k}\Big)^k\exp\Big(-\frac{1}{C}kt\Big)\leq \exp\Big(-\frac{1}{2C}kt\Big)
\end{align*}
if $t\geq C'\log(en/k)$.
\end{proof}

However for a general isotropic log-concave vector without unconditionality assumption 
we may bound $\Pr(X_{i_1}\geq t,\ldots,X_{i_k}\geq t)$ only by $\exp(-\sqrt{k}t/C)$ for $t\geq C$. 
This suggests that we should rather expect bound  $\exp(-\sqrt{k}t/C)$ than 
$\exp(-kt/C)$. If we try union bound
as in the proof of Proposition \ref{uncond} it will work only for $t\geq C\sqrt{k}\log(en/k)$.

Another approach may be based on the exponential concentration. We say that a vector $X$ satisfies
exponential concentration inequality with a constant $\alpha$ if for any Borel set $A$,
\[
\Pr(X\in A+\alpha tB_2^n)\geq 1-\exp(-t)\quad \mbox{ if }\
 \Pr(X\in A)\geq \frac{1}{2} \mbox{ and }t>0. 
\]

\begin{prop}
If coordinates of $n$-dimensional vector $X$ have mean zero and variance
one and $X$ satisfies exponential concentration inequality with a constant $\alpha\geq 1$
then
\[
\Pr(X_k^*\geq t)\leq \exp\Big(-\frac{1}{3\alpha}\sqrt{k}t\Big)
\quad \mbox{ for } t\geq 8\alpha\log\Big(\frac{en}{k}\Big).
\]
\end{prop}

\begin{proof}
Since $\Var(X_i)=1$ we have $\Pr(|X_i|\leq 2)\geq 1/2$ so $\Pr(|X_i|\geq 2+t)\leq \exp(-t/\alpha)$
for $t>0$. Let $\mu$ be the distribution of $X$. Then the set
\[
A(t)=\Big\{x\in\er^n\colon\ \#\{i\colon |x_i|\geq t\}< \frac{k}{2}\Big\}
\]
has measure $\mu$  at least $1/2$ for  $t\geq 4\alpha\log(en/k)$ -- indeed we have
for such $t$
\begin{align*}
1-\mu(A(t))&=\Pr\Big(\sum_{i=1}^n\ind_{\{|X_i|\geq t\}}\geq \frac{k}{2}\Big)\leq 
\frac{2}{k}\Ex\Big(\sum_{i=1}^n\ind_{\{|X_i|\geq t\}}\Big)
\\
&\leq \frac{2n}{k}\exp\Big(-\frac{t}{2\alpha}\Big)
\leq \frac{2n}{k}\Big(\frac{en}{k}\Big)^{-2}
\leq \frac{1}{2}.
\end{align*}

Let $A=A(4\alpha\log(en/k))$.
If $z=x+y\in A+\sqrt{k}sB_2^n$ then less than $k/2$ of $|x_i|$'s
are bigger than  $4\alpha \log(en/k)$ and less than $k/2$ of $|y_i|$'s are
bigger than $\sqrt{2}s$, so
\[
\Pr\Big(X_k^*\geq 4\alpha\log\Big(\frac{en}{k}\Big)+\sqrt{2}s\Big)\leq 
1-\mu(A+\sqrt{k}sB_2^n)\leq \exp(-\frac{1}{\alpha}\sqrt{k}s).
\]
\end{proof}

For log-concave vectors it is known that exponential inequality is equivalent to several other
functional inequalities such as Cheeger's and spectral gap -- see \cite{Mi} for a detailed
discussion and recent results. Strong open conjecture due
to Kannan, Lov\'asz and Simonovits \cite{KLS} states that every isotropic log-concave vector 
satisfies  Cheeger's (and therefore also exponential) inequality with a uniform constant. 
The conjecture however is wide open -- a recent result of Klartag \cite{Kl} 
shows that in the unconditional case
KLS conjecture holds up to $\log n$ constant (see also \cite{Hu} for examples 
of nonproduct distributions 
that satisfy spectral gap inequality with uniform constants). Best known upper 
bound for Cheeger's constant 
for general isotropic log-concave measure is $n^{\alpha}$ for some $\alpha \in (1/4,1/2)$ 
(see \cite{Mi}). 

The main result of this paper states that despite the approach via union bound or 
exponential concetration
fails the natural estimate for order statistics is valid. Namely we have

\begin{thm}
\label{estorder}
Let $X$ be $n$-dimensional log-concave isotropic vector. Then
\[ 
\Pr(X_k^*\geq t)\leq \exp\Big(-\frac{1}{C}\sqrt{k}t\Big) \quad 
\mbox{for }t\geq C\log\Big(\frac{en}{k}\Big).
\]
\end{thm}

Our approach is based on the suitable estimate of moments of the process
$N_X(t)$, where
\[
N_X(t):=\sum_{i=1}^n\ind_{\{X_i\geq t\}}\quad t\geq 0.
\]

\begin{thm}
\label{estN}
For any isotropic log-concave vector $X$ and $p\geq 1$ we have
\[
\Ex(t^2N_X(t))^p\leq (Cp)^{2p} \quad \mbox{ for } t\geq C\log\Big(\frac{nt^2}{p^2}\Big).
\]
\end{thm}

We postpone a long and bit technical proof till the last section of the paper. Let us only mention
at this point that it  is based on two ideas. One is the Paouris large deviation inequality
\eqref{conc_mass} and another is an observation that if we restrict a log-concave distribution 
to a convex set it is still log-concave.

\begin{proof}[Proof of Theorem \ref{estorder}]
Observe that $X^*_k\geq t$ implies that $N_X(t)\geq k/2$ or $N_{-X}(t)\geq k/2$ and
vector $-X$ is also isotropic and log-concave.
So by Theorem \ref{estN} and Chebyshev's inequality we get
\[
\Pr(X^*_k\geq t)\leq \Big(\frac{2}{k}\Big)^p\big(\Ex N_X(t)^p+\Ex N_{-X}(t)^p\big)
\leq 2 \Big(\frac{Cp}{t\sqrt{k}}\Big)^{2p} 
\]
provided that $t\geq C\log(nt^2/p^2)$. So it is enough to take $p=\frac{1}{eC}t\sqrt{k}$
and notice that the restriction on $t$ follows by the assumption that $t\geq C\log(en/k)$.
\end{proof}

As we already noticed one of the main tools in the proof of Theorem \ref{estN} is the Paouris 
concentration of mass.  One may however also do the opposite and derive large deviations for the
Euclidean norm of $X$ from our estimate of moments of $N_X(t)$ and observation that
the distribution of $UX$ is again log-concave and isotropic for any rotation $U$.
More precisely  the following statement holds.

\begin{prop}
Suppose that $X$ is a random vector in $\er^n$ such that for some constants $A_1,A_2<\infty$
and any $U\in O(n)$,
\[
\Ex\big(t^2N_{UX}(t)\big)^l\leq (A_1l)^{2l}
\quad \mbox{ for } t\geq A_2,\ l\geq \sqrt{n}.
\]
Then
\[
\Pr(|X|\geq t\sqrt{n})\leq \exp\Big(-\frac{1}{CA_1}t\sqrt{n}\Big) \quad \mbox{ for } 
t\geq \max\{CA_1,A_2\}.
\]
\end{prop}

\begin{proof}
Let us fix $t\geq A_2$. 
H\"older's inequality gives that for any $U_1,\ldots,U_n\in O(n)$
\[
\Ex\prod_{i=1}^lN_{U_iX}(t)\leq \Big(\prod_{i=1}^l\Ex N_{U_iX}(t)^l\Big)^{1/l}
\leq \Big(\frac{A_1l}{t}\Big)^{2l}
\quad \mbox{ for } l\geq \sqrt{n}.
\] 
Now let $U_1,\ldots,U_l$ be independent random rotations in $O(n)$ (distributed according to
the Haar measure) then for $l\geq \sqrt{n}$
\begin{align*}
\Big(\frac{A_1l}{t}\Big)^{2l}&\geq \Ex_X\Ex_U\prod_{i=1}^lN_{U_iX}(t)=
\Ex_X(\Ex_{U_1} N_{U_1X}(t))^l
=\Ex_X(n\Pr_Y(\langle X,Y\rangle\geq t))^l
\\
&=n^l\Ex_X(\Pr_Y(|X|Y_1\geq t))^l,
\end{align*}
where $Y$ is a random vector uniformly distributed on $S^{n-1}$. 
Since $Y_1$ is symmetric, $\Ex Y_1^2=1/n$ and
$\Ex Y_1^4\leq C/n^2$ we get that $\Pr(Y_1^2\geq \frac{1}{4n})\geq 1/C_1$ which gives
\[
\Pr(|X|\geq 2t\sqrt{n})\leq \Ex_X\big(C_1\Pr_Y(|X|Y_1\geq t)\big)^l\leq
\Big(\frac{C_1A_1^2l^2}{t^2n}\Big)^l.
\]
To conclude the proof it is enough to take $l=\big\lceil \frac{1}{\sqrt{eC_1}A_1}\sqrt{n}t\big\rceil$.
\end{proof}

\section{Concentration of $l_r$ norms}

The aim of this section is to derive Paouris--type estimates  for concentration of 
$\|X\|_r=(\sum_{i=1}^n |X_i|^r)^{1/r}$. We start with presenting two simple examples.

\medskip

\noindent
{\bf Example 1.}
Let coordinates of $X$ be independent exponential r.v's, then
\[
(\Ex\|X\|_r^r)^{1/r}=(n\Ex|X_1|^r)^{1/r}\geq \frac{1}{C}rn^{1/r} \quad
\mbox{ for }r\in [1,\infty),
\]
\[
\Ex\|X\|_{\infty}\geq \frac{1}{C}\log n
\]
and 
\[
(\Ex\|X\|_r^p)^{1/p}\geq (\Ex|X_1|^p)^{1/p}\geq \frac{p}{C} \quad \mbox{ for }p\geq 2,r\geq 2. 
\]
It is also known that in the independent exponential case weak and strong moments are comparable
\cite{La}, hence for $r\geq 2$ 
\begin{align*}
(\Ex\|X\|_r^r)^{1/r}&=
 \Big(\Ex\sup_{\|a\|_{r'}\leq 1}\Big|\sum_{i}a_iX_i\Big|^r\Big)^{1/r}
\\
&\leq (\Ex\|X\|_r^2)^{1/2}+C\sup_{\|a\|_{r'}\leq 1}\Big(\Ex\Big|\sum_{i}a_iX_i\Big|^r\Big)^{1/r}
\\
&\leq (\Ex\|X\|_r^2)^{1/2}+Cr\sup_{\|a\|_{r'}\leq 1}\Big(\Ex\Big|\sum_{i}a_iX_i\Big|^2\Big)^{1/2}
\leq (\Ex\|X\|_r^2)^{1/r}+Cr.
\end{align*}
Therefore  we get
\[
(\Ex\|X\|_r^p)^{1/p}\geq (\Ex\|X\|_r^2)^{1/2}\geq \frac{1}{C}rn^{1/r} \quad
\mbox{ for } p\geq 2 \mbox{ and } n\geq C^r.
\]

\medskip
\noindent
{\bf Example 2.} For $1\leq r\leq 2$ let $X$ be such isotropic random vector 
that $Y=(X_1+\ldots+X_n)/\sqrt{n}$
has an exponential distribution then by H\"older's inequality  $\|X\|_r\geq n^{1/r-1/2}Y$ and
\[
(\Ex\|X\|_r^p)^{1/p}\geq n^{1/r-1/2}\|Y\|_p\geq \frac{1}{C}n^{1/r-1/2}p
\quad \mbox{ for } p\geq 2,\ 1\leq r\leq 2.
\]

\medskip

The examples above show that the best we can hope is
\begin{equation}
\label{momsmallr}
(\Ex\|X\|_r^p)^{1/p}\leq C(n^{1/r}+n^{1/r-1/2}p)\quad \mbox{ for } 1\leq r\leq 2,
\end{equation}
\begin{equation}
\label{momlarger}
(\Ex\|X\|_r^p)^{1/p}\leq C(rn^{1/r}+p) \quad \mbox{ for }  r\in [2,\infty)
\end{equation}
and
\begin{equation}
\label{mominfty}
(\Ex\|X\|_{\infty}^p)^{1/p}\leq C(\log n+p).
\end{equation}

Or in terms  of tails,
\begin{equation}
\label{tailsmallr}
\Pr(\|X\|_r\geq t)\leq \exp\Big(-\frac{1}{C}tn^{1/2-1/r}\Big) \quad 
\mbox{ for } t\geq Cn^{1/r},\ r\in [1,2],
\end{equation}
\begin{equation}
\label{taillarger}
\Pr(\|X\|_r\geq t)\leq \exp\Big(-\frac{1}{C}t\Big) \quad 
\mbox{ for } t\geq Crn^{1/r},\ r\in[2,\infty)
\end{equation}
and
\begin{equation}
\label{tailinfty}
\Pr(\|X\|_{\infty}\geq t)\leq \exp\Big(-\frac{1}{C}t\Big) \quad \mbox{ for } t\geq C\log n.
\end{equation}

Case $r\in [1,2]$ is a simple consequence of Paouris theorem.

\begin{prop}
Estimates \eqref{momsmallr} and \eqref{tailsmallr} hold for all isotropic log-concave vectors $X$.
\end{prop}

\begin{proof}
We have $\|X\|_r\leq n^{1/r-1/2}\|X\|_2$ by H\"older's inequality, hence \eqref{momsmallr}
(and therefore also \eqref{tailsmallr}) immediately follows by Paouris result.
\end{proof}

Case $r=\infty$ is also very simple

\begin{prop}
\label{inftynorm}
Estimates \eqref{mominfty} and \eqref{tailinfty} hold for all isotropic log-concave vectors $X$.
\end{prop}

\begin{proof}
We have
\[
\Pr(\|X\|_{\infty}\geq t)\leq \sum_{i=1}^n\Pr(|X_i|\geq t)\leq n\exp(-t/C).
\]
\end{proof}

What is left is the case $2<r<\infty$ --
we would like to obtain \eqref{taillarger} and \eqref{momlarger}. We almost get it -- except 
that  constants explode when $r$ approaches 2.

\begin{thm}
\label{thm_larger}
For any $\delta>0$ there exist constants $C_1(\delta),C_2(\delta)\leq C\delta^{-1/2}$ 
such that for any $r\geq 2+\delta$,
\[
\Pr(\|X\|_r\geq t)\leq \exp\Big(-\frac{1}{C_1(\delta)}t\Big) \quad \mbox{ for } 
t\geq C_1(\delta)rn^{1/r}
\]
and
\[
(\Ex\|X\|_r^p)^{1/p}\leq 
C_2(\delta)\Big(rn^{1/r}+p\Big) \quad \mbox{ for } p\geq 2.
\]
\end{thm}

The proof of Theorem \ref{thm_larger} is based on the following slightly more precise estimate.

\begin{prop}
\label{estlarger}
For $r>2$  we have
\[
\Pr(\|X\|_r\geq t)\leq \exp\Big(-\frac{1}{C}\Big(\frac{r-2}{r}\Big)^{1/r}t\Big) \quad \mbox{ for } 
t\geq C\Big(rn^{1/r}+\Big(\frac{r}{r-2}\Big)^{1/r}\log n\Big)
\]
or in terms of moments
\[
(\Ex\|X\|_r^p)^{1/p}\leq 
C\Big(rn^{1/r}+\Big(\frac{r}{r-2}\Big)^{1/r}(\log n+p)\Big) \quad \mbox{ for } p\geq 2.
\]
\end{prop}

\begin{proof}
%%If $n\geq e^r$  and $t\geq Crn^{1/r}$ then $tn^{-1/r}\geq t/e\geq C\log n/e$ and
%%\[
%%\Pr(\|X\|_r\geq t)\leq \Pr(\|X\|_\infty\geq tn^{-1/r})\leq \exp(-\frac{t}{C}) 
%%\]
%%by Proposition \ref{inftynorm}.
%%Thus in the sequel we will assume that $n\geq e^r$.

Let $s=\lfloor\log_2 n\rfloor$ we have
\[
\|X\|_r^r=\sum_{i=0}^s|X_i^*|^r\leq \sum_{k=0}^s 2^k|X_{2^k}^{*}|^r.
\]
Theorem \ref{estorder} yields
\begin{equation}
\label{estr1}
\Pr\Big(|X_k^*|^r\geq C_3^r\log^r\Big(\frac{en}{k}\Big)+t^r\Big)\leq \exp\Big(-\frac{1}{C}\sqrt{k}t\Big)
\quad \mbox{ for }t>0.
\end{equation}

Observe that
\[
\sum_{k=0}^s 2^k\log^r(en2^{-k})\leq Cn\sum_{j=1}^{\infty}j^r2^{-j}\leq (Cr)^rn.
\]
Thus for $t_1,\ldots,t_k\geq 0$ we get
\[
\Pr\Big(\|X\|_r\geq C\Big(rn^{1/r}+\Big(\sum_{k=0}^s t_k\Big)^{1/r}\Big)\Big)\leq 
\Pr\Big(\sum_{k=0}^sY_k\geq \sum_{k=0}^s t_k\Big)
%%\leq \sum_{k=0}^s\Pr\Big(Y_k\geq t_k\Big), 
\]
where
\[
Y_k:=2^k\big(|X_{2^k}^{*}|^r-C_3^r\log^r(en2^{-k})\big).
\]
Hence by \eqref{estr1}
\begin{align*}
\Pr\Big(\|X\|_r\geq C\Big(rn^{1/r}+\Big(\sum_{k=0}^s t_k\Big)^{1/r}\Big)\Big)&\leq 
\sum_{k=0}^s\Pr\Big(Y_k\geq t_k\Big)
\\
&\leq \sum_{k=0}^s\exp\Big(-\frac{1}{C}2^{\frac{k}{2}-\frac{k}{r}}t_k^{1/r}\Big).
\end{align*}

Fix $t>0$ and choose $t_k$ such that $t=2^{k/2-k/r}t_k^{1/r}$, then
\[
\sum_{k=0}^st_k=t^r\sum_{k=0}^s2^{\frac{k(2-r)}{2}}\leq t^r\big(1-2^{\frac{r-2}{2}}\big)^{-1}\leq 
Ct^{r}\frac{r}{r-2},
\]
so we get
\[
\Pr\Big(\|X\|_r\geq C\Big(rn^{1/r}+t\Big(\frac{r}{r-2}\Big)^{1/r}\Big)\Big)
\leq (\log_2 n+1)\exp\Big(-\frac{1}{C}t\Big).
\]
\end{proof}

\begin{proof}[Proof of Theorem \ref{thm_larger}]
Observe that $(\frac{r}{2-r})^{1/r}\leq C\delta^{-1/2}$ for $r\geq 2+\delta$ and
$\log n\leq rn^{1/r}$ and apply Proposition \ref{estlarger}.
\end{proof}

\section{Proof of Theorem \ref{estN}}

Our crucial tool will be the following result.

\begin{prop}
\label{conditional}
Let $X$ be an isotropic log-concave $n$-dimensional random vector, 
$A=\{X\in K\}$ where $K$ is a convex set in $\er^n$ such that $0<\Pr(A)\leq 1/e$. Then
\begin{equation}
\label{cond1}
\sum_{i=1}^n\Pr(A\cap\{X_i\geq t\})\leq C_1\Pr(A)\Big(t^{-2}\log^2(\Pr(A))+ne^{-t/C_1}\Big)\quad 
\mbox{ for } t\geq C_1.
\end{equation}
Moreover for $1 \leq u\leq \frac{t}{C_2}$,
\begin{equation}
\label{cond2}
\#\{i\leq n\colon \Pr(A\cap\{X_i\geq t\})\geq e^{-u}\Pr(A)\}\leq
\frac{C_2u^2}{t^2}\log^2(\Pr(A)). 
%%\quad \mbox{ for } C_2\leq u\leq \frac{t}{C_2}.
\end{equation}
\end{prop}

\begin{proof}
Let $Y$ be a random vector distributed as vector $X$ conditioned on $A$ that is
\[
\Pr(Y\in B)=\frac{\Pr(A\cap\{X\in B\})}{\Pr(A)}=\frac{\Pr(X\in B\cap K)}{\Pr(X\in K)}.
\]
Notice that in particular for any set $B$, $\Pr(X\in B)\geq \Pr(A)\Pr(Y\in B)$.

Vector $Y$ is log-concave, but no longer isotropic. Since this is only a matter of permutation of
coordinates we may assume that $\Ex Y_1^2\geq \Ex Y_2^2\geq \ldots\geq \Ex Y_n^2$.

For $\alpha>0$ let 
\[
m=m(\alpha)=\#\{i\colon \Ex Y_i^2\geq \alpha\}.
\]
We have $\Ex Y_1^2\geq\ldots\geq \Ex Y_m^2\geq \alpha$. Hence by the Paley-Zygmund inequality,
\[
\Pr\Big(\sum_{i=1}^m Y_i^2\geq \frac{1}{2}\alpha m\Big)\geq 
\Pr\Big(\sum_{i=1}^m Y_i^2\geq \frac{1}{2}\Ex\sum_{i=1}^m Y_i^2\Big)\geq 
\frac{1}{4}\frac{\Ex(\sum_{i=1}^m Y_i^2)^2}{(\Ex \sum_{i=1}^m Y_i^2)^2}\geq \frac{1}{C}.
\]
This implies that
\[
\Pr\Big(\sum_{i=1}^m X_i^2\geq \frac{1}{2}\alpha m\Big)\geq \frac{1}{C}\Pr(A). 
\]
However by the result of Paouris,
\[
\Pr\Big(\sum_{i=1}^m X_i^2\geq \frac{1}{2}\alpha m\Big)\leq 
\exp\Big(-\frac{1}{C_3}\sqrt{m\alpha}\Big)
\quad \mbox{ for }\alpha\geq C_3.
\]
So for $\alpha\geq C_3$, $\exp(-\frac{1}{C_3}\sqrt{m\alpha})\geq \Pr(A)/C$ and we get that
\begin{equation}
\label{estm}
m(\alpha)=\#\{i\colon \Ex Y_i^2\geq \alpha\}\leq \frac{C_4}{\alpha}\log^2(\Pr(A))\quad 
\mbox{ for }\alpha\geq C_3.
\end{equation}

We have
\[
\frac{\Pr(A\cap\{X_i\geq t\})}{\Pr(A)}=
\Pr(Y_i\geq t)\leq \exp\Big(1-\frac{t}{C(\Ex Y_i^2)^{1/2}}\Big)
\]
and \eqref{cond2} follows by \eqref{estm}.

Take $t\geq \sqrt{C_3}$ and let $k_0$ be a nonnegative integer such that 
$2^{-k_0}t\geq \sqrt{C_3}\geq 2^{-k_0-1}t$.
Define
\[
I_0=\{i\colon \Ex Y_i^2\geq t^2\},\quad I_{k_0+1}=\{i\colon \Ex Y_i^2<4^{-k_0}t^2\}
\]
and
\[
I_j=\{i\colon 4^{-j}t^2\leq \Ex Y_i^2<4^{1-j}t^2\}\quad j=1,2,\ldots,k_0.
\]
By \eqref{estm} we get
\[
\#I_j\leq C_44^jt^{-2}\log^2\Pr(A) \quad \mbox{for }j=0,1,\ldots,k_0
\]
and obviously $\#I_{k_0+1}\leq n$. Moreover for $i\in I_j$, $j\neq 0$,
\[
\Pr(Y_j\geq t)\leq \Pr\Big(\frac{Y_j}{(\Ex Y_j^2)^{1/2}}\geq 2^{j-1}\Big)\leq 
\exp\Big(1-\frac{1}{C}2^j\Big).
\]
Thus
\begin{align*}
\sum_{i=1}^{n}\Pr(Y_i\geq t)&=\sum_{j=0}^{k_0+1}\sum_{i\in I_j}\Pr(Y_i\geq t)
\leq \#I_0+e\sum_{j=1}^{k_0+1}\#I_j\exp\Big(-\frac{1}{C}2^j\Big)
\\
&\leq C_4\bigg(t^{-2}\log^2\Pr(A)\bigg(1+e\sum_{j=1}^{k_0}2^{2j}\exp\Big(-\frac{1}{C}2^j\Big)\bigg)+
ene^{-t/C}\bigg)\\
&\leq C_1\Big(t^{-2}\log^2\Pr(A)+ne^{-t/C_1}\Big).
\end{align*}
To finish the proof of \eqref{cond1} it is enough to observe that
\[
\sum_{i=1}^n\Pr(A\cap\{X_i\geq t\})=\Pr(A)\sum_{i=1}^n\Pr(Y_i\geq t).
\]
\end{proof}

The following two examples show that estimate \eqref{cond1} is close to be optimal.

\medskip

\noindent
{\bf Example 1.} Take $X_1,X_2,\ldots,X_n$ to be independent symmetric
exponential random variables with variance 1 and $A=\{X_1\geq \sqrt{2}\}$
Then $\Pr(A)=\frac{1}{2e}$ and
\[
\sum_{i=2}^n\Pr(A\cap\{X_i\geq t\})=\Pr(A)\sum_{i=2}^n\Pr(X_i\geq t)=
(n-1)\Pr(A)\exp(-t/\sqrt{2}),
\]
therefore the factor $ne^{-t/C}$ in \eqref{cond1} is necessary.

\medskip

\noindent
{\bf Example 2.} Take $A=\{X_1\geq t,\ldots,X_k\geq t\}$ then 
\[
\sum_{i=1}^n\Pr(A\cap\{X_i\geq t\})\geq k\Pr(A).
\]
So improvement of the factor $t^{-2}\Pr(A)\log^2\Pr(A)$ in \eqref{cond1} would imply in particular
a better estimate of $\Pr(X_1\geq t,\ldots,X_k\geq t)$ than $\exp(-\frac{1}{C}\sqrt{k}t)$,
which we do not know if there is possible to obtain.

\begin{proof}[Proof of Theorem \ref{estN}]
We have $N_X\leq n$, so statement is obvious if $t\sqrt{n}\leq Cp$,  in the sequel we will
assume that $t\sqrt{n}\geq 10p$.

 Let $C_1$ and $C_2$ be as in 
Proposition \ref{conditional} -- increasing $C_i$ if necessary we may assume
that $\Pr(X_1\geq t)\leq e^{-t/C_i}$ for $t\geq C_i$ and $i=1,2$.
Let us fix $p\geq 1$ and $t\geq C\log(\frac{nt^2}{p^2})$, then $t\geq \max\{C_1,4C_2\}$
and $t^2ne^{-t/C_1}\leq p^2$ if $C$ is large enough.
Let $l$ be a positive integer such that
\[
p\leq l\leq 2p \quad \mbox{ and } l=2^k \mbox{ for some  integer }k.
\]
Since $(\Ex (N_X(t))^p)^{1/p}\leq (\Ex (N_X(t))^l)^{1/l}$ it is enough to show 
that
\[
\Ex(t^2N_X(t))^l\leq (Cl)^{2l}.
\]
Recall that by our assumption on $p$, we have $t\sqrt{n}\geq 5l$.

To shorten the notation let
\[
B_{i_1,\ldots,i_s}=\{X_{i_1}\geq t,\ldots,X_{i_s}\geq t\}\quad \mbox{ and }\quad
B_{\emptyset}=\Omega.
\]
Define
\[
m(l):=\Ex N_X(t)^l=\Ex\Big(\sum_{i=1}^n\ind_{\{X_i\geq t\}}\Big)^l=
\sum_{i_1,\ldots,i_l=1}^{n}\Pr(B_{i_1,\ldots,i_l}),
\]
we need to show that
\begin{equation}
\label{toshow1}
m(l)\leq \Big(\frac{Cl}{t}\Big)^{2l}.
\end{equation}

We devide the sum in $m(l)$ into several parts.
Let 
$j_1\geq 2$ be such integer that 
\[
2^{j_1-2}< \log\Big(\frac{nt^2}{l^2}\Big)\leq 2^{j_1-1}.
\]
We set
\[
I_{0}=\big\{(i_1,\ldots,i_l)\in\{1,\ldots,n\}^l\colon \Pr(B_{i_1,\ldots,i_l})> e^{-l}\big\},
\] 
\[
I_{j}=\big\{(i_1,\ldots,i_l)\in\{1,\ldots,n\}^l\colon 
\Pr(B_{i_1,\ldots,i_l})\in (e^{-2^{j}l},e^{-2^{j-1}l}]
\big\}\quad   0<j<j_1
\]
and
\[
I_{j_1}=\big\{(i_1,\ldots,i_l)\in\{1,\ldots,n\}^l\colon 
\Pr(B_{i_1,\ldots,i_l})\leq e^{-2^{j_1-1}l}\big\}.
\]
Since $\{1,\ldots,n\}^l=\bigcup_{j=0}^{j_1}I_j$ we get $m(l)=\sum_{j=0}^{j_1}m_j(l)$, where
\[
m_j(l):=\sum_{(i_1,\ldots,i_l)\in I_j}\Pr(B_{i_1,\ldots,i_l})\quad
\mbox{ for } j_0\leq j\leq j_1.
\]

Bound for $m_{j_1}(l)$ is easy -- namely
since $\#I_{j_1}\leq n^l$ we have
\[
\sum_{(i_1,\ldots,i_l)\in I_{j_1}}\Pr(B_{i_1,\ldots,i_l})\leq 
n^le^{-2^{j_1-1}l}\leq \Big(\frac{l}{t}\Big)^{2l}.
\]

To estimate $m_{0}(l)$ define first for
$I\subset\{1,\ldots,n\}^l$ and $1\leq s\leq l$,
\[
P_sI=\{(i_1,\ldots,i_s)\colon (i_1,\ldots,i_l)\in I \mbox{ for some }
i_{s+1},\ldots,i_l\}.
\]
By Proposition \ref{conditional} we get for $s=1,\ldots,l-1$
\begin{align*}
\sum_{(i_1,\ldots,i_{s+1})\in P_{s+1}I_{0}}&\Pr(B_{i_1,\ldots,i_{s+1}})\leq
\sum_{(i_1,\ldots,i_{s})\in P_{s}I_{0}}\sum_{i_{s+1}=1}^n
\Pr(B_{i_1,\ldots,i_s}\cap\{X_{i_{s+1}}\geq t\})
\\
&\leq
C_1\sum_{(i_1,\ldots,i_{s})\in P_{s}I_{0}}
\Pr(B_{i_1,\ldots,i_{s}})(t^{-2}\log^2\Pr(B_{i_1,\ldots,i_{s}})+ne^{-t/C_1}).
\end{align*}
Observe that we have $\Pr(B_{i_1,\ldots,i_s})\geq e^{-l}$ for $(i_1,\ldots,i_s)\in P_sI_0$ and
recall that $t^2ne^{-t/C_1}\leq p^2\leq 4l^2$, hence
\[
\sum_{(i_1,\ldots,i_{s+1})\in P_{s+1}I_{0}}\Pr(B_{i_1,\ldots,i_{s+1}})
\leq 5C_1t^{-2}l^2\sum_{(i_1,\ldots,i_{s})\in P_{s}I_{j_0}}
\Pr(B_{i_1,\ldots,i_{s}}).
\]
So, by easy induction we obtain
\begin{align*}
m_{0}(l)&=\sum_{(i_1,\ldots,i_l)\in I_{0}}\Pr(B_{i_1,\ldots,i_l})\leq 
(5C_1t^{-2}l^2)^{l-1}
\sum_{i_1\in P_1I_0}\Pr(B_{i_1})
\\
&\leq (5C_1t^{-2}l^2)^{l-1}ne^{-t/C_1}\leq \Big(\frac{Cl}{t}\Big)^{2l}.
\end{align*}

Now it comes the most involved part -- estimating $m_j(l)$ for $0<j<j_1$, that is based on
suitable bounds for $\#I_j$. We will need the following
simple combinatorial lemma.

\begin{lem}
\label{combf}
Let $l_0\geq l_1\geq\ldots\geq l_s$ be a fixed sequence of positive integers and
\[
{\cal F}=\Big\{f\colon\{1,2,\ldots,l_0\}\rightarrow\{0,1,2,\ldots,s\}\colon\ 
\forall_{1\leq i\leq s}\ \#\{r\colon f(r)\geq i\}\leq l_i\Big\}.
\]
Then 
\[
\#{\cal F}\leq\prod_{i=1}^s\Big(\frac{el_{i-1}}{l_i}\Big)^{l_i}.
\]
\end{lem}

\begin{proof}[Proof of Lemma \ref{combf}]
Notice that any function $f\colon\{1,2,\ldots,l_0\}\rightarrow\{0,1,2,\ldots,s\}$ is determined by
the sets $A_i=\{r\colon f(r)\geq i\}$ for $i=0,1,\ldots,s$. Take $f\in {\cal F}$, obviously 
$A_0=\{1,\ldots,l_0\}$. If the set $A_{i-1}$ of cardinality $a_{i-1}\leq l_{i-1}$ is already chosen then the set $A_i\subset A_{i-1}$ of cardinality at most $l_i$ may be
chosen in
\[
\binom{a_{i-1}}{0}+\binom{a_{i-1}}{1}+\ldots+\binom{a_{i-1}}{l_i}\leq
\binom{l_{i-1}}{0}+\binom{l_{i-1}}{1}+\ldots+\binom{l_{i-1}}{l_i}\leq 
\Big(\frac{el_{i-1}}{l_i}\Big)^{l_i}
\]
ways.
\end{proof}

We come back to the proof of Theorem \ref{estN}. Fix $0<j<j_1$,
let  $r_1$ be a positive integer such that
\[
2^{r_1}< \frac{t}{C_2}\leq 2^{r_1+1}.
\] 
For $(i_1,\ldots,i_l)\in I_j$
we define a function $f_{i_1,\ldots,i_l}\colon \{1,\ldots,l\}\rightarrow \{j,j+1,\ldots,r_1\}$
by the formula
\[
f_{i_1,\ldots,i_l}(s)=
\left\{
\begin{array}{ll}
j &\mbox{ if }\Pr(B_{i_1,\ldots,i_s})\geq \exp(-2^{j+1})\Pr(B_{i_1,\ldots,i_{s-1}}),
\\
r &\mbox{ if }
\exp(-2^{r+1})\leq \frac{\Pr(B_{i_1,\ldots,i_s})}{\Pr(B_{i_1,\ldots,i_{s-1}})}<\exp(-2^{r}),\
j<r<r_1,
\\
r_1 &\mbox{ if }\Pr(B_{i_1,\ldots,i_s})< \exp(-2^{r_1})\Pr(B_{i_1,\ldots,i_{s-1}}).
\end{array}
\right.
\]
Notice that for all $i_1$, $\Pr(X_{i_1}\geq t)\leq e^{-t/C_2}< \exp(-2^{r_1})\Pr(B_{\emptyset})$,
so $f_{i_1,\ldots,i_l}(1)=r_1$ for all $i_1,\ldots,i_l$.

Put
\[
{\cal F}_j:=\big\{f_{i_1,\ldots,i_l}\colon\ (i_1,\ldots,i_l)\in I_j\big\}. 
\]
For $f=f_{i_1,\ldots,i_l}\in {\cal F}_j$ and $r>j$, we have
\[
\exp(-2^{j}l)< \Pr(B_{i_1,\ldots,i_l})< \exp(-2^r\#\{s\colon f(s)\geq r\}),
\]
so
\begin{equation}
\label{est_l_r}
\#\{s\colon f(s)\geq r\}\leq 2^{j-r}l=:l_r.
\end{equation}
Observe that the above inequality holds also for $r=j$.
We have $l_{r-1}/l_r=2$ and $\sum_{r=j+1}^{r_1}l_r\leq l$ so by Lemma \ref{combf}
we get 
\[
\#{\cal F}_j
\leq \prod_{r=j+1}^{r_1}\Big(\frac{el_{r-1}}{l_r}\Big)^{l_r}\leq e^{2l}.
\]

Now fix $f\in {\cal F}_j$ we will estimate the cardinality of the set
\[
I_j(f):=\{(i_1,\ldots,i_l)\in I_j\colon\ f_{i_1,\ldots,i_j}=f\}.
\]
Put
\[
n_r:=\#\{s\in\{1,\ldots,l\}\colon f(s)=r\}\quad r=j,j+1,\ldots,r_1.
\]
We have 
\[
n_j+n_{j+1}+\ldots+n_{r_1}=l,
\]
moreover if $i_1,\ldots,i_{s-1}$ are fixed and 
$f(s)=r<r_1$ then $s\geq 2$ and by the second part of Proposition \ref{conditional}
(with $u=2^{r+1}\leq t/C_2$) $i_s$ may take at most
\[
\frac{4C_22^{2r}}{t^2}\log^2\Pr(B_{i_1,\ldots,i_{s-1}})\leq
\frac{4C_{2}2^{2(r+j)}l^2}{t^2}\leq
\frac{4C_2l^2}{t^2}\exp(2(r+j))=:m_r
\]
values. Thus
\[
\#I_j(f)\leq n^{n_{r_1}}\prod_{r=j}^{r_1-1}m_r^{n_r}=
n^{n_{r_1}}\Big(\frac{4C_2l^2}{t^2}\Big)^{l-n_{r_1}}\exp\Big(\sum_{r=j}^{r_1-1}2(r+j)n_r\Big).
\]
Observe that by previously derived estimate \eqref{est_l_r} we get
\[
n_r\leq l_r= 2^{j-r}l,
\]
hence
\[
\sum_{r=j}^{r_1-1}2(r+j)n_r\leq 2^{j+2}l\sum_{r=j}^{\infty}r2^{-r}\leq (C+2^{j-2})l.
\]
We also have
\[
n_{r_1}\leq 2^{j-r_1}l\leq\frac{2C_2}{t}2^jl\leq 
\frac{1}{\log(nt^2/(4l^2))}2^{j-3}l,
\]
where the last inequality holds since $t\geq C\log(nt^2/(4l^2))$ and $C$ may be taken arbitrarily large.
So we get that for any $f\in {\cal F}_j$,
\[
\#I_j(f)\leq \Big(\frac{Cl^2}{t^2}\Big)^l\Big(\frac{nt^2}{4l^2}\Big)^{n_{r_1}}\exp\big(2^{j-2}l\big)\leq 
\Big(\frac{Cl^2}{t^2}\Big)^l\exp\Big(\frac{3}{8}2^{j}l\Big).
\]

This shows that
\[
\#I_j\leq \#{\cal F}_j\cdot \Big(\frac{Cl^2}{t^2}\Big)^l\exp\Big(\frac{3}{8}2^{j}l\Big)\leq 
\Big(\frac{Cl^2}{t^2}\Big)^l\exp\Big(\Big(2+\frac{3}{8}2^{j}\Big)l\Big).
\]
Hence
\[
m_j(l)=\sum_{(i_1,\ldots,i_l)\in I_j}\Pr(B_{i_1,\ldots,i_l})\leq
\#I_j\exp(-2^{j-1}l)\leq \Big(\frac{Cl^2}{t^2}\Big)^l\exp\big(-2^{j-3}l\big).
\]
Therefore
\begin{align*}
m(l)&=m_0(l)+m_{j_1}(l)+\sum_{j=1}^{j_1-1}m_j(l)
\leq \Big(\frac{l}{t}\Big)^{2l}\Big(C^l+1+\sum_{j=1}^{\infty}C^l\exp\big(-2^{j-3}l\big)\Big)
\\
&\leq \Big(\frac{Cl}{t}\Big)^{2l}
\end{align*}
and \eqref{toshow1} holds.
\end{proof}

{\bf Acknowledgments.} This work was done while the author was taking part in the Thematic 
Program on Asymptotic Geometric Analysis at the  Fields Institute in Toronto. The author would
like to thank Rados{\l}aw Adamczak and Nicole Tomczak-Jaegermann for their helpful comments.

\end{document}